\documentclass{proc-l}

\usepackage{amscd}      
\usepackage{amssymb}
\usepackage{amsmath}
\usepackage{xypic}      
\LaTeXdiagrams          
\usepackage[all,v2]{xy}
\xyoption{2cell} \UseAllTwocells \xyoption{frame} \CompileMatrices
\allowdisplaybreaks[3]

\newtheorem{theorem}{Theorem}[section]
\newtheorem{lemma}[theorem]{Lemma}
\newtheorem{cor}[theorem]{Corollary}

\theoremstyle{definition}

\newtheorem{example}[theorem]{Example}

\theoremstyle{remark}
\newtheorem{remark}[theorem]{Remark}

\numberwithin{equation}{section}

\begin{document}

\title[Finite abelian gerbes over toric Deligne-Mumford stacks]{A Note on Finite Abelian Gerbes \\ over Toric Deligne-Mumford Stacks}

\author{Yunfeng Jiang}
\address{Department of Mathematics,  University of Utah, 155S 1400E JWB233, UT 84112}
\email{jiangyf@math.utah.edu}
\thanks{Received by the editors September 11, 2006, and, in revised form, May 8, 2007, June 10, 2007, 
October 11, 2007, and November 6, 2007.}



\begin{abstract}
Any toric Deligne-Mumford stack is a $\mu$-gerbe over the
underlying toric orbifold for a finite abelian group $\mu$. In
this paper we give a sufficient condition so that certain kinds of
gerbes over a toric Deligne-Mumford stack are again toric
Deligne-Mumford stacks.
\end{abstract}

\maketitle

\section{Introduction}
Let $\mathbf{\Sigma}:=(N,\Sigma,\beta)$ be a stacky fan of
$rank(N)=d$ as defined in \cite{BCS}. If there are $n$
one-dimensional cones in the fan $\Sigma$, then modelling the
construction of toric varieties \cite{Cox},\cite{F},  the toric
Deligne-Mumford stack $\mathcal{X}(\mathbf{\Sigma})=[Z/G]$ is a
quotient stack, where $Z=\mathbb{C}^{n}-V$, the close subvariety
$V\subset \mathbb{C}^{n}$ is  determined by the ideal $J_{\Sigma}$
generated by $\{\prod_{\rho_{i}\nsubseteq\sigma}z_{i}:
\sigma\in\Sigma\}$ and $G$ acts on $Z$ through the map $\alpha: G
\longrightarrow (\mathbb{C}^{\times})^{n}$ in the following exact
sequence determined by the stacky fan (see \cite{BCS}):
\begin{equation}\label{stack}
1\longrightarrow \mu\longrightarrow
G\stackrel{\alpha}{\longrightarrow}
(\mathbb{C}^{\times})^{n}\longrightarrow T\longrightarrow 1. \
\end{equation}
Let $\overline{G}=Im(\alpha)$, then $[Z/\overline{G}]$ is the
underlying toric orbifold $\mathcal{X}(\mathbf{\Sigma_{red}})$.
The toric Deligne-Mumford stack  $\mathcal{X}(\mathbf{\Sigma})$ is
a $\mu$-gerbe over $\mathcal{X}(\mathbf{\Sigma_{red}})$.

Let $\mathcal{X}(\mathbf{\Sigma})$ be a toric Deligne-Mumford
stack associated to the stacky fan $\mathbf{\Sigma}$. Let $\nu$ be
a finite abelian group, and let $\mathcal{G}$ be a $\nu$-gerbe
over $\mathcal{X}(\mathbf{\Sigma})$. We give a sufficient
condition so that  $\mathcal{G}$ is also a toric Deligne-Mumford
stack.  We have the following theorem:

\begin{theorem}
Let $\mathcal{X}(\mathbf{\Sigma})$ be a toric Deligne-Mumford
stack with stacky fan $\mathbf{\Sigma}$. Then every $\nu$-gerbe
$\mathcal{G}$ over $\mathcal{X}(\mathbf{\Sigma})$ is induced by a
central extension
$$1\longrightarrow\nu\longrightarrow \widetilde{G}\longrightarrow
G\longrightarrow 1,$$ i.e., we have a Cartesian diagram:
$$\xymatrix{
~\mathcal{G}\dto^{}\rto^{} &~\mathcal{B}\widetilde{G}\dto^{}\\
~\mathcal{X}(\mathbf{\Sigma})\rto^{} &~\mathcal{B}G .}$$
\end{theorem}

In general, the $\nu$-gerbe $\mathcal{G}$ is not a toric
Deligne-Mumford stack. But if the central extension is abelian,
then we have:

\begin{cor}
If the $\nu$-gerbe $\mathcal{G}$ is induced from an abelian
central extension, it is a toric Deligne-Mumford stack.
\end{cor}

This small note is organized  as follows. In Section 2 we
construct the new toric Deligne-Mumford stack from an abelian
central extension and prove the main results . In Section 3 we
give an example of  $\nu$-gerbe over a toric Deligne-Mumford
stack.

In this paper, by an $orbifold$ we mean a smooth Deligne-Mumford
stack with trivial stabilizers at the  generic points.

\section{The Proof of Main Results}

We refer the reader to \cite{BCS} for the construction and
notation of toric Deligne-Mumford stacks. For the general theory
of stacks, see \cite{BEFFGK}.

Let $\mathbf{\Sigma}:=(N,\Sigma,\beta)$ be a stacky fan. From
Proposition 2.2 in \cite{BCS}, we have the following exact
sequences:
$$0\longrightarrow DG(\beta)^{*}\longrightarrow \mathbb{Z}^{n}\stackrel{\beta}{\longrightarrow}
N\longrightarrow Coker(\beta)\longrightarrow
0,$$
$$0\longrightarrow N^{*}\longrightarrow \mathbb{Z}^{n}\stackrel{\beta^{\vee}}{\longrightarrow}
DG(\beta)\longrightarrow Coker(\beta^{\vee})\longrightarrow 0,$$
where $\beta^{\vee}$ is the Gale dual of $\beta$. As a
$\mathbb{Z}$-module, $\mathbb{C}^{\times}$ is divisible, so it is
an injective $\mathbb{Z}$-module and hence the functor
$Hom_{\mathbb{Z}}(-,\mathbb{C}^{\times})$ is exact.  We get the
exact sequence:
$$1\longrightarrow Hom_{\mathbb{Z}}(Coker(\beta^{\vee}),\mathbb{C}^{\times})\longrightarrow
Hom_{\mathbb{Z}}(DG(\beta),\mathbb{C}^{\times})\longrightarrow
Hom_{\mathbb{Z}}(\mathbb{Z}^{n},\mathbb{C}^{\times})$$
$$\longrightarrow
Hom_{\mathbb{Z}}(N^{*},\mathbb{C}^{\times})\longrightarrow 1.$$
Let
$\mu:=Hom_{\mathbb{Z}}(Coker(\beta^{\vee}),\mathbb{C}^{\times})$,
we have the exact sequence (\ref{stack}). Let $\Sigma(1)=n$ be the
set of one dimensional cones in $\Sigma$ and $V\subset
\mathbb{C}^{n}$ the closed subvariety defined by the ideal
generated by
$$J_{\Sigma}=\left<\prod_{\rho_{i}\nsubseteq \sigma} z_{i}:
\sigma \in \Sigma \right>.$$ Let $Z:=\mathbb{C}^{n}\setminus
V$. From \cite{Cox}, the complex codimension of $V$ in
$\mathbb{C}^{n}$ is at least $2$. The toric Deligne-Mumford stack
$\mathcal{X}(\mathbf{\Sigma})=[Z/G]$ is the quotient stack where
the action of $G$ is through the map $\alpha$ in (\ref{stack}).

\begin{lemma}\label{keylemma}
If $Codim_{\mathbb{C}}(V,\mathbb{C}^{n})\geq 2$, then
$H^{1}(Z,\nu)=H^{2}(Z,\nu)=0$, where $\nu$ is a finite abelian
group.
\end{lemma}
\begin{proof}
Consider the following exact sequence:
\begin{eqnarray}
0& \longrightarrow & H^{0}_{V}(\mathbb{C}^{n},\nu)\longrightarrow
H^{0}(\mathbb{C}^{n},\nu)\longrightarrow
H^{0}(Z,\nu)\longrightarrow  \nonumber \\
&\longrightarrow & H^{1}_{V}(\mathbb{C}^{n},\nu)\longrightarrow
H^{1}(\mathbb{C}^{n},\nu)\longrightarrow
H^{1}(Z,\nu)\longrightarrow \nonumber \\
&\longrightarrow & H^{2}_{V}(\mathbb{C}^{n},\nu)\longrightarrow
\cdots \nonumber
\end{eqnarray}
Since $Codim_{\mathbb{C}}(V,\mathbb{C}^{n})\geq 2$, so the real
codimension is at least $4$ and $H^{i}_{V}(\mathbb{C}^{n},\nu)=0$
for $i=1,2,3$, so from the exact sequence and
$H^{i}(\mathbb{C}^{n},\nu)=0$ for all $i>0$ we prove the lemma.
\end{proof}
\subsection{The Proof of Theorem 1.1.}
Consider the following diagram
$$\xymatrix{
Z\rto\dto & pt\dto\\
[Z/G]\rto^\pi & \ \mathcal{B}G}$$ which is Cartesian.  Consider
the Leray spectral sequence for the fibration $\pi$:
$$H^{p}(\mathcal{B}G,R^{q}\pi_{*}\nu)\Longrightarrow H^{p+q}([Z/G],\nu).$$
We compute
$$H^{2}([Z/G],\nu)=\bigoplus_{p+q=2}H^{p}(\mathcal{B}G,R^{q}\pi_{*}\nu).$$
First we have that $R^{q}\pi_{*}\nu=[H^{q}(Z,\nu)/G]$. There are
three cases.
\begin{enumerate}
\item When $p=2, q=0$, $R^{0}\pi_{*}\nu=\nu$ because $Z$ is
connected, so
$$H^{p}(\mathcal{B}G,R^{q}\pi_{*}\nu)=H^{2}(\mathcal{B}G,\nu);$$
\item When $p=1, q=1$, $R^{1}\pi_{*}\nu=[H^{1}(Z,\nu)/G]$, so
$$H^{p}(\mathcal{B}G,R^{q}\pi_{*}\nu)=H^{1}(\mathcal{B}G,H^{1}(Z,\nu)),$$
and by Lemma
\ref{keylemma}, $H^{1}(Z,\nu)=0$, so we have
$H^{p}(\mathcal{B}G,R^{q}\pi_{*}\nu)=0$; \item When $p=0, q=2$,
$R^{2}\pi_{*}\nu=[H^{2}(Z,\nu)/G]$, so
$$H^{p}(\mathcal{B}G,R^{q}\pi_{*}\nu)=H^{0}(\mathcal{B}G,H^{2}(Z,\nu)),$$
also from Lemma \ref{keylemma}, $H^{2}(Z,\nu)=0$, we have
$H^{p}(\mathcal{B}G,R^{q}\pi_{*}\nu)=0$.
\end{enumerate}
So we get that
$$H^{2}([Z/G],\nu)\cong H^{2}(\mathcal{B}G,\nu).$$
Since for the finite abelian group $\nu$, the $\nu$-gerbes  are
classified by the second cohomology group with coefficient in the
group $\nu$. Theorem 1.1 is proved. $\square$

\subsection{The Proof of Corollary 1.2.}
Let $\mathcal{X}(\mathbf{\Sigma})=[Z/G]$. The $\nu$-gerbe
$\mathcal{G}$ over $[Z/G]$ is induced from  a $\nu$-gerbe
$\mathcal{B}\widetilde{G}$ over $\mathcal{B}G$ in the following
central extension
$$1\longrightarrow\nu\longrightarrow
\widetilde{G}\stackrel{\varphi}{\longrightarrow} G\longrightarrow
1,$$ where $\widetilde{G}$ is an abelian group. So the pullback
gerbe over $Z$ under the map $Z\longrightarrow [Z/G]$ is trivial.
So we have
$$\mathcal{G}=\mathcal{B}\widetilde{G}\times_{\mathcal{B}G}[Z/G]=[Z/\widetilde{G}].$$
The stack $[Z/\widetilde{G}]$
is this $\nu$-gerbe $\mathcal{G}$ over $[Z/G]$. Consider the
commutative diagram:
\begin{equation}\label{graph}
\vcenter{\xymatrix{~
\widetilde{G}\dto_{\widetilde{\alpha}}\rto^{\varphi} & G \dto^{\alpha} \\
(\mathbb{C}^{\times})^{n}\rto^{\cong} &
(\mathbb{C}^{\times})^{n}}}
\end{equation}
where $\alpha$ is the map in (\ref{stack}). So we have the
following exact sequences:
$$1\longrightarrow
\nu\longrightarrow
ker(\widetilde{\alpha})\stackrel{}{\longrightarrow}
\mu\longrightarrow 1$$ and
$$1\longrightarrow
ker(\widetilde{\alpha})\longrightarrow
\widetilde{G}\stackrel{\widetilde{\alpha}}{\longrightarrow}
(\mathbb{C}^{\times})^{n}\longrightarrow T\longrightarrow 1,$$
where $T$ is the torus of the simplicial toric variety
$X(\Sigma)$. Since the abelian groups $\widetilde{G}$, $G$ and
$(\mathbb{C}^{\times})^{n}$ are all locally compact topological
groups, taking  Pontryagin duality and Gale dual, we have the
following diagrams:
\[
\begin{CD}
0 @ >>>N^{*}@ >>> \mathbb{Z}^{n}@ >{\beta^{\vee}}>> DG(\beta) @
>>> Coker(\beta^{\vee})@>>> 0\\
&& @VV{}V@VV{id}V@VV{p_{\varphi}}V@VVV \\
0@ >>>\widetilde{N}^{*}@
>{}>>\mathbb{Z}^{n}@
>{(\widetilde{\beta})^{\vee}}>> DG(\widetilde{\beta}) @>>>Coker((\widetilde{\beta})^{\vee})
@>>>0,
\end{CD}
\]

\[
\begin{CD}
0 @ >>>DG(\widetilde{\beta})^{*}@ >>> \mathbb{Z}^{n}@
>{\widetilde{\beta}}>> \widetilde{N} @
>>> Coker(\widetilde{\beta})@>>> 0\\
&& @VV{}V@VV{id}V@VV{}V@VVV \\
0 @ >>>DG(\beta)^{*}@
>{}>>\mathbb{Z}^{n}@
>{\beta}>> N @>>>Coker(\beta) @>>>0,
\end{CD}
\]
where $p_{\varphi}$ is induced by $\varphi$ in (\ref{graph}) under
the  Pontryagin duality. Suppose $\widetilde{\beta}:
\mathbb{Z}^{n}\longrightarrow \widetilde{N}$ is given by
$\{\widetilde{b}_{1},\cdots,\widetilde{b}_{n}\}$, then
$\mathbf{\widetilde{\Sigma}}:=(\widetilde{N},\Sigma,\widetilde{\beta})$
is a new stacky fan. The toric Deligne-Mumford stack
$\mathcal{X}(\mathbf{\widetilde{\Sigma}})=[Z/\widetilde{G}]$ is
the $\nu$-gerbe $\mathcal{G}$ over $\mathcal{X}(\mathbf{\Sigma})$.
$\square$

\begin{remark}
From Proposition 4.6 in \cite{BN},  any Deligne-Mumford stack is a
$\nu$-gerbe over an orbifold for a finite group $\nu$. Our results
are the toric case of that general result.

In particular, from a stacky fan
$\mathbf{\Sigma}=(N,\Sigma,\beta)$. Let
$\mathbf{\Sigma_{red}}=(\overline{N},\Sigma,\overline{\beta})$ be
the $reduced$ stacky fan, where $\overline{N}$ is the abelian
group $N$ modulo torsion, and
$\overline{\beta}:\mathbb{Z}^{n}\longrightarrow \overline{N}$ is
given by $\{\overline{b}_{1},\cdots,\overline{b}_{n}\}$ which are
the images of $\{b_{1},\cdots,b_{n}\}$ under the natural
projection $N\longrightarrow \overline{N}$. Then the toric
orbifold $\mathcal{X}(\mathbf{\Sigma_{red}})=[Z/\overline{G}]$.
From (\ref{stack}), let $\overline{G}=Im(\alpha)$, then we have
the following exact sequences:
$$1\longrightarrow \overline{G}\stackrel{}{\longrightarrow}
(\mathbb{C}^{\times})^{n}\longrightarrow T\longrightarrow 1,$$
$$1\longrightarrow
\mu\longrightarrow
G\stackrel{}{\longrightarrow}\overline{G}\longrightarrow 1.$$ So
$G$ is an abelian central extension of $\overline{G}$ by $\mu$.
$\mathcal{X}(\mathbf{\Sigma})$ is a $\mu$-gerbe over the toric
orbifold $\mathcal{X}(\mathbf{\Sigma_{red}})$. Any  $\mu$-gerbe
over the toric orbifold coming from an abelian central extension
is a toric Deligne-Mumford stack. This is a special case of the
main results and is the toric case of rigidification construction
in \cite{ACV}.
\end{remark}

\begin{remark}
From the proof of Corollary 1.2 we see that if a $\nu$-gerbe over
$\mathcal{X}(\mathbf{\Sigma})$ comes from a gerbe over
$\mathcal{B}G$ and the central extension is abelian, then we can
construct a new toric Deligne-Mumford stack.
\end{remark}

\section{An Example}
\begin{example}
Let $\Sigma$ be the complete fan of the projective line,
$N=\mathbb{Z}\oplus \mathbb{Z}/3\mathbb{Z}$, and $\beta:
\mathbb{Z}^{2}\longrightarrow \mathbb{Z}\oplus
\mathbb{Z}/3\mathbb{Z}$ be given by the vectors
$\{b_{1}=(1,0),b_{2}=(-1,1)\}$. Then
$\mathbf{\Sigma}=(N,\Sigma,\beta)$ is a stacky fan.  We compute
that $(\beta)^{\vee}: \mathbb{Z}^{2}\longrightarrow
DG(\beta)=\mathbb{Z}$ is given by the matrix [3,3]. So we get the
following exact sequence:
\begin{equation}\label{gerbe1}
1\longrightarrow \mu_{3}\longrightarrow
\mathbb{C}^{\times}\stackrel{[3,3]^{t}}{\longrightarrow}
(\mathbb{C}^{\times})^{2}\longrightarrow
\mathbb{C}^{\times}\longrightarrow 1 \
\end{equation}
The toric Deligne-Mumford stack
$\mathcal{X}(\mathbf{\Sigma})=[\mathbb{C}^{2}-\{0\}/\mathbb{C}^{\times}]$,
where the action is given by
$\lambda(x,y)=(\lambda^{3}x,\lambda^{3}y)$. So
$\mathcal{X}(\mathbf{\Sigma})$ is the nontrivial $\mu_{3}$-gerbe
over $\mathbb{P}^{1}$ coming from the canonical line bundle over
$\mathbb{P}^{1}$.  Let $\mathcal{G}\longrightarrow
\mathcal{X}(\mathbf{\Sigma})$ be a $\mu_{2}$-gerbe such that it
comes from the $\mu_{2}$-gerbe over
$\mathcal{B}\mathbb{C}^{\times}$ given by the central extension
\begin{equation}\label{gerbe2}
1\longrightarrow \mu_{2}\longrightarrow
\mathbb{C}^{\times}\stackrel{(\cdot)^{2}}{\longrightarrow}
\mathbb{C}^{\times}\longrightarrow 1. \
\end{equation}
From  the sequence (\ref{gerbe1}) and (\ref{gerbe2}), we have:
$$
1\longrightarrow \mu_{3}\otimes \mu_{2}\longrightarrow
\mathbb{C}^{\times}\stackrel{[6,6]^{t}}{\longrightarrow}
(\mathbb{C}^{\times})^{2}\longrightarrow
\mathbb{C}^{\times}\longrightarrow 1 .\
$$
The Pontryagin dual of
$\mathbb{C}^{\times}\stackrel{[6,6]^{t}}{\longrightarrow}
(\mathbb{C}^{\times})^{2}$ is $(\widetilde{\beta})^{\vee}:
\mathbb{Z}^{2}\longrightarrow \mathbb{Z}$ which is given by the
matrix $[6,6]$. Taking Gale dual we have:
$$\widetilde{\beta}: \mathbb{Z}^{2}\longrightarrow \mathbb{Z}\oplus \mathbb{Z}_{6},$$
which is given by the vectors
$\{\widetilde{b}_{1}=(1,0),\widetilde{b}_{2}=(-1,1)\}$. Let
$\mathbf{\widetilde{\Sigma}}=(\widetilde{N},\Sigma,\widetilde{\beta})$
be the new stacky fan, then we have the toric Deligne-Mumford
stack
$\mathcal{X}(\mathbf{\widetilde{\Sigma}})=[\mathbb{C}^{2}-\{0\}/\mathbb{C}^{\times}]$,
where the action is given by
$\lambda(x,y)=(\lambda^{6}x,\lambda^{6}y)$. So
$\mathcal{X}(\mathbf{\widetilde{\Sigma}})$ is the canonical
$\mu_{6}$-gerbe over $\mathbb{P}^{1}$.

If the $\mu_{2}$-gerbe over $\mathcal{B}\mathbb{C}^{\times}$ is
given by the central extension
\begin{equation}\label{gerbe3}
1\longrightarrow \mu_{2}\longrightarrow
\mathbb{C}^{\times}\times\mu_{2}\stackrel{\alpha}{\longrightarrow}
\mathbb{C}^{\times}\longrightarrow 1, \
\end{equation}
where $\alpha$ is given by the matrix $[1,0]$. Then from
(\ref{gerbe1}) and (\ref{gerbe3}), we have
$$
1\longrightarrow \mu_{3}\otimes \mu_{2}\longrightarrow
\mathbb{C}^{\times}\times\mu_{2}\stackrel{\varphi}{\longrightarrow}
(\mathbb{C}^{\times})^{2}\longrightarrow
\mathbb{C}^{\times}\longrightarrow 1 ,\
$$
where $\varphi$ is given by the matrix $\left[
\begin{array}{cc}
  3&0 \\
  3&0\\
\end{array}
\right]. $ The Pontryagin dual of $\varphi$ is:
$(\widetilde{\beta}^{'})^{\vee}: \mathbb{Z}^{2}\longrightarrow
\mathbb{Z}\oplus \mathbb{Z}_{2}$ which is given by the transpose
of the above matrix. Taking Gale dual we get
$$\widetilde{\beta}^{'}: \mathbb{Z}^{2}\longrightarrow \mathbb{Z}\oplus \mathbb{Z}_{3}\oplus \mathbb{Z}_{2},$$
which is given by the vectors
$\{\widetilde{b}_{1}=(1,0,0),\widetilde{b}_{2}=(-1,1,0)\}$. So
$\mathbf{\widetilde{\Sigma}}^{'}=(\widetilde{N}^{'},\Sigma,\widetilde{\beta}^{'})$
is a stacky fan. The toric Deligne-Mumford stack
$\mathcal{X}(\mathbf{\widetilde{\Sigma}}^{'})=[\mathbb{C}^{2}-\{0\}/\mathbb{C}^{\times}\times\mu_{2}]$,
where the action is
$(\lambda_{1},\lambda_{2})\cdot(x,y)=(\lambda_{1}^{3}x,\lambda_{1}^{3}y)$.
So $\mathcal{G}^{'}=\mathcal{X}(\mathbf{\widetilde{\Sigma}}^{'})$
is the trivial $\mu_{2}$-gerbe over $\mathcal{X}(\mathbf{\Sigma})$
and  $\mathcal{X}(\mathbf{\widetilde{\Sigma}})\ncong
\mathcal{X}(\mathbf{\widetilde{\Sigma}}^{'})$.
\end{example}

\section*{Acknowledgments}
I would like to thank the referee for nice comments about the
proof of main results. I thank my advisor Kai Behrend for
encouragements and Hsian-Hua Tseng for valuable discussions.

\bibliographystyle{amsplain}

\end{document}